\documentclass[12pt, a4paper, reqno]{amsart}

\usepackage{enumerate}
\usepackage{amssymb}
\usepackage{color}
\usepackage{graphics}
\usepackage{epsfig}

\newtheorem{lemma}{Lemma}[section]

\newtheorem{theorem}[lemma]{Theorem}
\newtheorem{proposition}[lemma]{Proposition}
\newtheorem{remark}[lemma]{Remark}
\newtheorem{definition}[lemma]{Definition}

\newcommand{\R}{\mathbb{R}}

\newcommand{\N}{\mathbb{N}}

\author[J. D. Garc\'{\i}a-Salda\~{n}a]{Johanna D. Garc\'{\i}a-Salda\~{n}a}
\address{Dept. de Matem\`{a}tiques \\
Universitat Aut\`{o}noma de Barcelona \\ Edifici C. 08193
Bellaterra, Barcelona. Spain} \email{johanna@mat.uab.cat}

\author[A. Gasull]{Armengol Gasull}
\address{Dept. de Matem\`{a}tiques.
Universitat Aut\`{o}noma de Barcelona. Edifici C. 08193 Bellaterra,
Barcelona. Spain} \email{gasull@mat.uab.cat}

\date{}
\dedicatory{} \commby{}

\begin{document}

\title[Period function and HBM]
{The period function \\ and  the harmonic balance method}
\begin{abstract}
In this paper we consider several families of potential
non-iso\-chro\-nous systems and study their associated period
functions. Firstly, we prove some properties of these  functions,
like their local behavior near the critical point or infinity, or
their global monotonicity. Secondly, we show that these properties
are also present when we approach to the same questions using the
Harmonic Balance Method.
\end{abstract}

\maketitle

\section{Introduction and main results}

Given a planar differential system having a continuum of periodic
orbits, its period function is defined as the function that
associates to each periodic orbit its period. To determine the
global behavior of this period function is an interesting problem in
the qualitative theory of differential equations either as a
theoretical question or due to its appearance  in many situations.
For instance, the period function is present in mathematical models
in physics or ecology, see \cite{CV, Rot, Wal} and the references
therein; in the study of some bifurcations \cite[pp. 369-370]{CH};
or to know the number of solutions of some associated boundary value
problems, see \cite{Chi,Chi2}.

In particular, there are several works giving  criteria for
determining the  monotonicity of the period function associated with
some systems, see \cite{Chi,F-G-G,G-G-V,Sa,Zhao} and the references
therein. Results about non monotonous period functions have also
recently appeared, see for instance \cite{GGJ,Gas-xin,MaVi}.

The so-called $N$-th order Harmonic Balance Method (HBM) consists on
approximating the periodic solutions of a non-linear differential
equation by using truncated Fourier series of order $N$. It is
mainly applied with practical purposes, although in many cases there
is no a theoretical justification. In most of the applications this
method is used to approach  isolated periodic solutions, see for
instance \cite{Ga-Ga,Hu-Tang,Mi,Mi-Duf,Mi-lib,Mi-rac}.  Since the
HBM also provides an approximation of the angular frequency of the
searched periodic solution, it can be also used to get  its period.

Hence, applying  the HBM  to systems of differential equations
having a continuum of periodic orbits  we can obtain  approximations
of the corresponding period functions. The main goal of this paper
is to illustrate this last assertion trough the study of several
concrete planar systems. This approach is also used for instance in
\cite{B2,B1,Mic-book2}. A main difference among these works and our
paper is that we also carry out a detailed analytic study of the
involved period functions.

More specifically, in this work we will consider several families of
planar potential systems, $\ddot x=f(x),$ having  continua of
periodic orbits. We will study analytically their corresponding
period functions  and we will see that the approximations of the
period functions obtained using the $N$-th order HBM, for $N=1$,
keep the essential properties of the actual period functions: local
behavior near the critical point and infinity, monotonicity,
oscillations,...  For the case of the Duffing oscillator we also
consider $N=2$ and $3$. In particular, the method that we introduce
using resultants gives  an analytic  way  to deal with the 3rd order
HBM, answering question (iv) in \cite[p. 180]{Mic-book2}.

First, we focus in the following two families of potential
differential systems:
\begin{equation}\label{sis}
\left\{\begin{array}{l}
\dot x=-y,\\
\dot y=x+x^{2m-1},\qquad m\in \N\quad {\rm and}\quad m\geq 2,
\end{array}\right.
\end{equation}
and
\begin{equation}\label{sis2}
\left\{\begin{array}{lll}
\dot{x}=y,\\
\dot{y}=-\frac{x}{(x^{2}+k^{2})^{m}}, \qquad k\in\R\setminus \{0\},
\,\,  m\in [1,\infty).
\end{array}\right.
\end{equation}
Each system of these families has a continuum of periodic orbits
around the origin. Thus, we can talk about its  periodic function
$T$ which associates to each periodic orbits passing trough
$(x,y)=(A,0)$ its period $T(A)$. In addition, we will denote by
$T_N(A)$ the approximation to $T(A)$ by using $N$-th order HBM; see
Section \ref{bhm} for the precise definition of $T_N(A)$.

System~\eqref{sis} is an extension of the Duffing-harmonic
oscillator which corresponds to the case $m=2$. The case $m=2$  has
been studied by many authors, see
\cite{Hu-Tang,Lim-Wu,M-J-B-G,Mi-Duf}. The exact period function of
this particular system is given as an  elliptic function and so it
is easier to obtain analytic properties of $T$. Our analytic
 study  is  valid for all integers $m\ge2.$

System \eqref{sis2} with $m=1$ and by taking the limit $k\rightarrow
0$ is equivalent to the second order differential equation $x\ddot
x+1=0$, which is studied in \cite{Mi} as a model of plasma physics.
Thus, system~\eqref{sis2} can be seen as an extension of the
singular second order differential equation $x\ddot x+1=0$. In a
forthcoming paper we explore the relationship between the periodic
solutions of \eqref{sis2} with $m=1$ and their corresponding periods
with the solutions of the limiting case $x\ddot x+1=0.$

We have chosen these two families due to their simplicity and
because, as we will see, their corresponding period functions are
monotonous, being the first one decreasing and the second one
increasing.

For the first family~\eqref{sis}, in addition to the monotonicity of
$T$, we perform a more detailed study of some properties of $T$.
More precisely, we give the behavior of $T$ near to the origin and
at infinity  and we compare them with the results obtained by the
HBM.

\begin{theorem}\label{teo 1.1}
System~\eqref{sis} has a global center at the origin and its period
function $T$ is decreasing. Moreover, at $A=0$,
\begin{align}\label{tay}
T(A)=2\pi\left(1-\frac{(2m-1)!!}{(2m)!!}\,A^{2m-2}+{S}(m)\,A^{4m-4}+O(A^{6m-6})\right),
\end{align}
where ${S}(m)=\frac{(2m-1)(4m-1)!!}{m(4m)!!}
-\frac{(m-1)(2m-1)!!}{m(2m)!!}$; and
\begin{equation}\label{PFInfgen}
T(A)\sim B\left(\frac{1}{2m},\frac{1}{2}\right)\frac{2}{\sqrt{m}\,
A^{m-1}},\quad A\to\infty,
\end{equation}
where $B(\cdot,\cdot)$ is the Beta function.
\end{theorem}

\begin{proposition}\label{gen-Duf-HB}
By applying the first-order HBM to system \eqref{sis} we get the
decreasing function
\begin{equation}\label{fper1}
T_1(A)=\frac{2^m\pi}{\sqrt{\frac{(2m-1)!}{(m-1)!m!}A^{2m-2}+2^{2m-2}}}.
\end{equation}
 Moreover, at $A=0$,
$$
T_1(A)=2\pi\left(1-\frac{(2m-1)!!}{(2m)!!}A^{2m-2}
+\frac{3}{2}\left(\frac{(2m-1)!!}{(2m)!!}\right)^2
A^{4m-4}+O(A^{6m-6})\right),
$$
and
\begin{equation}
\label{PF1ogeninf}
T_1(A)\sim\frac{2^m\pi}{\sqrt{\frac{(2m-1)!}{(m-1)!m!}}\,
A^{m-1}},\quad A\to\infty.
\end{equation}
\end{proposition}

By  Theorem~\ref{teo 1.1}, we know that the period function $T$ of
system~\eqref{sis} is decreasing. Proposition~\ref{gen-Duf-HB}
asserts that this property is already present in its first order
approximation obtained with the HBM. Additionally, we can see that
the first and second terms of the Taylor series at $A=0$ of $T(A)$
and $T_1(A)$ coincide, while the third one is different.
Furthermore, from \eqref{PFInfgen} and \eqref{PF1ogeninf} it follows
that $T(A)$ and $T_1(A)$ have similar behaviors at infinity.

In the case of the Duffing-harmonic oscillator ($m=2$
in~\eqref{sis}) we will apply the $N$-th order HBM, $N=2,3,$ for
computing the approximations $T_N(A)$ of the period function $T(A)$,
see Section~\ref{Duff}. We prove that $T(A)-T_N(A)=O(A^{2N+4})$ at
$A=0.$ We believe that similar results hold for  \eqref{sis} with
$m>2$, nevertheless, for the sake of shortness, do not study this
question it this paper. We also will see that the approximations
$T_N(A), N=1,2,3,$ at infinity  become  sharper by increasing $N$.

\smallskip

For the family~\eqref{sis2} we have similar results. We only will
deal with the global behaviors of $T$ and $T_1$ skipping the study
of these functions near zero and infinity.

\begin{theorem}\label{tem 1.3}
System~\eqref{sis2} has a center at the origin and its period
function $T$ is increasing. Moreover, the center is global for $m=1$
and non-global otherwise.
\end{theorem}

\begin{proposition}\label{genMI-BH}
By applying the first-order HBM to system \eqref{sis2} we obtain the
increasing function
\begin{equation}\label{PFfmic}
T_1(A)=2\pi\sqrt{\sum_{j=0}^{m}\left(\frac{1}{2}\right)^{2j}{m\choose
j}{ {2j+1}\choose j}k^{2(m-j)} A^{2j}}.
\end{equation}
\end{proposition}

Note that again, as in system~\eqref{sis}, with the first-order HBM
we obtain that $T_1(A)$ and $T(A)$ have the same monotonicity
behavior.

In Section \ref{final} we consider the family of polynomial
potential systems
\begin{equation}\label{sisCri}
\left\{\begin{array}{l}
\dot{x}=-y,\\
\dot{y}=x+k\, x^3+x^5, \quad k\in\R,
\end{array}\right.
\end{equation}
which for some values of  $k$ has a global center. In \cite[Thm. 1.1
(b)]{MaVi} it is  proved that the period function associated to the
global center at the origin has at most one oscillation. Joining
this result  with a similar study that the one made for
system~\eqref{sis} at the origin and at infinity, that we will omit
for the sake of shortness,  we obtain:

\begin{theorem} Consider system~\eqref{sisCri} with  $k\in(-2,\infty)$.
 Let $T$ be the period function associated to the origin, which  is a global
 center. Then:
\begin{itemize}
 \item [$(i)$] The function $T$ is
monotonous decreasing for $k\ge0$.
 \item [$(ii)$] The function $T$  starts increasing, until a maximum (a critical period)
 and then decreases towards zero, for $k\in(-2,0)$.
\item [$(iii)$] At the origin
\[
T(A)= 2\pi- \frac{3}{4}k\pi A^2+\frac{57k^2-80}{128}\pi A^4+O(A^6),
\]
and at infinity
\[
T(A)\sim \frac{2{ B}(\frac16,\frac12)}{\sqrt 3}\frac
1{A^2}\approx\frac{8.4131}{A^2}.
\]
\end{itemize}
\end{theorem}

We prove:

\begin{proposition}\label{profinal} By applying the first-order HBM to the family \eqref{sisCri}
we get:
$$
T_1(A)=\frac{8\pi}{\sqrt{16+12kA^2+10A^4}}.
$$
In particular,
\begin{itemize}
 \item [$(i)$] The function $T_1(A)$ is  decreasing for $k\geq0$.
 \item [$(ii)$] The function $T_1(A)$ starts increasing, has a
 maximum and then decreases towards zero, for $k\in(-2,0)$.
\item [$(iii)$] At the origin
\[
T_1(A)= 2\pi- \frac{3}{4}k\pi A^2+\frac{54k^2-80}{128}\pi
A^4+O(A^6),
\]
and at infinity
\[
T_1(A)\sim \frac{4\pi\sqrt{10}}{5}\frac
1{A^2}\approx\frac{7.9477}{A^2}.\]
\end{itemize}
\end{proposition}

Once more, we can see that the function $T_1(A)$ obtained by
applying the first order HBM captures and reproduces quite well the
actual behavior of $T(A).$

\begin{remark} In fact, the shape of the function $T_1(A)$ for $k\in(-2,0)$ does not vary
until  $k=-2\sqrt{10}/3\approx -2.107$. For $k\le -2\sqrt{10}/3$, it
is no more defined for all $A\in\R.$ Somehow, this phenomenon
reflects the fact that for $k\le-2$ the center in not global.
Notice, that for $k<-2$, system \eqref{sisCri} has three centers.
\end{remark}

Motivated by all our results, in Section~\ref{secPoten} we study the
relationship between the Taylor series of $T(A)$ and $T_N(A)$ with
$N=1,2$ at $A=0$ for an arbitrary smooth potential.
\smallskip

When the  system has a center and its period function is constant,
then the center is called {\it isochronous}. The problem about the
existence and characterization of isochronous center has also  been
extensively studied, see \cite{Chr-De,C-G-M,C-M-V,Grav,Ma-Ro-T}. To
end this introduction we want to comment that we have not succeeded
in applying the HBM to detect isochronous potentials. We have unfold
in 1-parameter families one  of the simplest potential isochronous
systems, the one given by a rational potential function,
see~\cite{CV}. Our attempts to use the low order HBM to detect  the
value of the parameter that corresponds to the isochronous case have
not succeed.

The paper is organized as follows. In Section \ref{prelim} we give
some preliminary results which include a known result for studying
the monotonicity of the period function. Also we describe the $N$-th
order HBM. In Section~\ref{analyt} we prove our analytical results
about the monotonicity of the period function of systems \eqref{sis}
and \eqref{sis2} and their local behavior at the center and at
infinity, see Theorems~\ref{teo 1.1} and~\ref{tem 1.3}. In
Section~\ref{hbm proof} we prove Propositions~\ref{gen-Duf-HB}
and~\ref{genMI-BH}, both dealing with the HBM. In Section~\ref{Duff}
we focus  on the study of the Duffing-harmonic oscillator and we
also apply the 2-th order and 3-rd order HBM. Section \ref{final}
deals with the family of planar polynomial potential systems having
a non-monotonous period function. Finally, Section~\ref{secPoten}
studies the local behavior near zero of $T(A)$ and $T_N(A)$ with
$N=1,2$, of an arbitrary smooth potential system.

\section{Preliminary results}\label{prelim} This section is divided in two parts. The
first one is devoted to recall some  definitions, as well as, to
give the framework for the study of the period function of
\eqref{sis} and \eqref{sis2} from an analytical point of view. In
the second one we will give the description of the $N$-order
Harmonic Balance Method, which we will apply in our second analysis
of the period function.

\subsection{Definitions and some analytical tools}
The systems studied in this paper are all potential systems,
\begin{equation}\label{hamil}
\left\{\begin{array}{lll}
\dot{x}=-y,\\
\dot{y}=\,F'(x),
\end{array}\right.
\end{equation}
with associated Hamiltonian function  $H(x,y)=y^2/2+F(x)$, where
$F:\Omega\subset \R \to \R$ is a real smooth function, $F(0)=0$ and
$0\in \Omega,$ an open real interval.

Let $p_0$ be a singular point of \eqref{hamil}. It is said that
$p_0$ is a {\it center}  if there exists an open neighborhood $U$ of
$p_0$ such that each solution $\gamma(t)$ of  \eqref{hamil} with
$\gamma(0)\in U-\{p_0\}$ defines a periodic orbit $\gamma$
surrounding  $p_0$.  The largest neighborhood $\mathcal{P}$ with
this property is called the $\textit{period annulus}$ of $p_0$. If
$\Omega=\R$ and $\mathcal{P}=\R^2$, then $p_0$ is called a {\it
global center}.

The following result characterizes  systems \eqref{hamil} having
global  centers.
\begin{lemma}
\label{lemcent} If $F(x)$  has a  minimum at $0$, then system
\eqref{hamil} has a center at the origin. Moreover, the center is
global if and only if $F'(x)\ne0$ for all $x\ne0$ and $F(x)$ tends
to infinity when $|x|$ does.
\end{lemma}

Suppose that  \eqref{hamil} has a center with period annulus
$\mathcal{P}$. For each periodic orbit $\gamma\in\mathcal{P}$ we
define $T(\gamma)$ to be the period of $\gamma$. Thus, the map
$$
T:\mathcal{P}\rightarrow\R_+,\qquad \gamma\mapsto T(\gamma),
$$
is called the $\textit{period func\-tion}$ associated with
$\mathcal{P}$. It is said that the map $T$ is {\it monotone
increasing} (respectively {\it monotone decreasing}) if for each
couple of periodic orbits $\gamma_0$ and $\gamma_1$ in
$\mathcal{P}$, with $\gamma_0$ in the interior of bounded region
surrounded by $\gamma_1$, it holds that $T(\gamma_1)-T(\gamma_0)>0$
(respectively $<0$). When $T$ is constant, then the center is called
 {\it isochronous center}.

If  we  fix a transversal section $\Sigma$ to $\mathcal{P}$ and we
take a parametrization $\sigma(A)$ of $\Sigma$ with
$A\in(0,A^*)\subset\R_+$, then we can denote by $\gamma_A$  the
periodic orbit passing through $\sigma(A)$ and by $T(A)$ its period.
That is, we have the map $T:(0,A^*)\rightarrow\R_+$, $A\mapsto
T(A)$.  When $T$ is not monotonous then either it is constant or it
has local maxima or minima. The isolated zeros of $T'(A)$ are called
{\it critical periods}. It is not difficult to prove that the number
of critical periods does not depend neither of $\Sigma$ nor of its
parametrization.

Next, we will recall two results about some properties of the period
function $T$  which we will apply in our study of the families
\eqref{sis}, \eqref{sis2} and \eqref{sisCri}. The first result is an
adapted version  to system \eqref{hamil}, of statement 3 of
\cite[Prop. 10]{F-G-G} and gives a criterion about the monotonicity
of $T$. The second one is an adapted version of \cite[Thm.
C]{C-G-M}, which will allow us to describe the behavior of $T$ at
infinity.

\begin{proposition}\label{propo10} Suppose that system \eqref{hamil} has a
center at the origin. Let $T$ be the period function associated
to the period annulus of the center. Then
\begin{itemize}
\item[$(i)$] If $F^{\prime}(x)^2-2F(x)F^{\prime \prime}(x)\geq 0$
 (not identically $0$) on $\Omega$, then  $T$ is increasing.
\item[$(ii)$] If $F^{\prime}(x)^2-2F(x)F^{\prime \prime}(x)\leq 0$
 (not identically $0$) on $\Omega$, then  $T$ is decreasing.
\end{itemize}
\end{proposition}

To state the second result, we need some previous constructions and
definitions.

Let $\gamma_h(t)=(x_h(t), y_h(t))$ be a periodic orbit of
\eqref{hamil} contained in $\mathcal{P}$ co\-rres\-pon\-ding to the
level set $\{H=h\}$. This orbit crosses the axis $y=0$ at the points
determined by $F(x_h(t))=h$. Since $F$ has a minimum at $x=0$, near
the origin the above equation has two solutions, one of them on
$x>0$ which will be denoted by $F_+^{-1}(h)$ and the other one on
$x<0$ which will be denoted by $F_-^{- 1}(h)$. We note that this
property remains for all
$h\in(0,h^*):=H(\mathcal{P})\setminus\{0\}$. For each $h>0$ we
define the function
\begin{equation}\label{lf-lg}
l_F(h)=F_+^{- 1}(h)-F_-^{- 1}(h)
\end{equation}
which gives the length of the projection to the $x$-axis of
$\gamma_h$. See Figure \ref{figuralf-lg}.

\begin{figure}[h]
\centering\epsfig{file=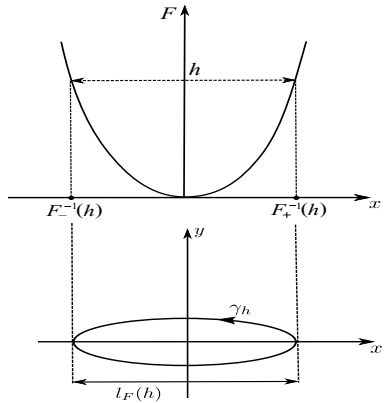,width=140pt,height=150pt}
\caption{Definition of $l_F$ for system \eqref{hamil}.}
\label{figuralf-lg}
\end{figure}

\begin{definition}
\label{asint} Given two real numbers $a$ and $M$, it is said that a
continuous function $g(x)$ has $Mx^a$ as  {\bf dominant term} of its
asymptotic expansion at $x=x_0\in \R\cup\{\infty\}$ if
$$ \lim_{x\to x_0}\frac{g(x)-Mx^a}{x^a} =0.$$
 This property is denoted by $g(x)\sim Mx^a$ at $x=x_0$.
\end{definition}

\begin{theorem} \label{teoC}  Assume that \eqref{hamil} has a global center at
the origin. Let $l_F(h)$ be
as in \eqref{lf-lg} and suppose
that $l^{\prime}_F (h)\sim Mh^a,$ at $h=\infty$ with $a>-1$ and
$M>0$. Then, the period function of \eqref{hamil} satisfies $
T(h)\sim Ch^{a+1/2} $ at $h=\infty$, where $C=\sqrt{2}M B(a+1,
3/2)$, and $ B(\cdot,\cdot)$ is the Beta function.
\end{theorem}

Next lemma computes the function $l_F(h)$ for system \eqref{sis}.

\begin{lemma}\label{lema}
\label{lemlf1} The function $l_F(h)$  associated to \eqref{sis}
satisfies that $\textit{l}_F(h)\sim 2(2mh)^{\frac{1}{2m}}$ at
$h=\infty$.
\end{lemma}

\begin{proof}
We start studying  the algebraic curve
$\mathcal{C}:=\{p(x,h)=F(x)-h=0\}$ at infinity. For that, we
consider the homogenization
\begin{equation}\label{pol-hom}
P(X,H,Z)= X^{2m}+mX^2Z^{2m-2}-2mHZ^{2m-1},
\end{equation}
of $p(x,h)$ in the real projective plane $\mathbb{RP}^2$. From
\eqref{pol-hom} it follows that  $[0:1:0]$ is the unique point at
infinity of $\mathcal{C}$, and  $\mathcal{C}$ in the chart that
contains such point is given by the set of zeros of the polynomial
$$\tilde p(\tilde x,\tilde z):=P(\tilde x,1,\tilde z)=
\tilde x^{2m}+m\tilde x^2z^{2m-2}- 2m\tilde z^{2m-1}.$$ For studying
$\mathcal{C}$ at infinity we will obtain a parametrization of it
close to the point $[0:1:0]$. As usual, we will use the  Newton
polygon  associated to $\tilde p$. The $\it{carrier}$ of $\tilde p$
is  ${\rm carr}(\tilde p)=\{(2m,0),(2,2m-2),(0,2m-1)\}$, whence the
Newton polygon is  the straight line joining $(2m,0)$ and $(0,2m-1)$
whose equation is
$$\tilde z=-\left(\frac{2m-1}{2m}\right)\tilde x + 2m-1.$$
In $\tilde p(\tilde x, \tilde  z)$ we replace $\tilde z=z_0 t^{2m}$
and $\tilde x=x_0 t^{2m-1}$,  then
$$
\tilde p(x_0 t^{2m-1},z_0 t^{2m})=m x_0^2 z_0^{2(m-1)} t^{2(2m^2-1)}
+(x_0^{2m}-2mz_0^{2m-1}) t^{2m(2m-1)}.
$$
For  $x_0$ fixed we consider
$$
\begin{array}{lll}
\phi(z_0,t)&=&(x_0^{2m}-2mz_0^{2m-1})t^{2m(2m-1)}+mx_0^2z_0^{2(m-1)}t^{2(2m^2-1)}
\\&=&t^{2m(2m-1)}\tilde\phi(z_0,t),
\end{array}
$$
 where $\tilde\phi(z_0,t)=x_0^{2m}-2mz_0^{2m-1}+mx_0^2z_0^{2(m-1)}t^{2(m-1)}$. It is clear that
$\tilde\phi(z_0^\ast,0)=0$ for $z_0^\ast$ solution of
$x_0^{2m}-2mz_0^{2m-1}=0$. Moreover
$$\frac{\partial\tilde\phi}{\partial z_0}(z_0^\ast,0)=-2m(2m-1)(z_0^\ast)^{2(m-1)}\neq 0.$$
From the implicit function theorem  there exists a function
$z_0(t):(\R,0)\rightarrow\R$  such that $z_0(0)=z_0^\ast$ and
$\phi(z_0(t),t)=0$ for $t\in(\R,0).$ Since $\phi(z_0(t),t)$ is an
analytic function, $z_0(t)$ also   is it. Hence we can write
$z_0(t)=c_0+c_1t+c_2t^2+\ldots$, moreover as $z_0(0)=z_0^\ast$ then
$$
z_0(t)=z_0^\ast+O(t).
$$
From  $x_0^{2m}-2mz_0^{2m-1}=0$ and the above equation it follows
that
$$x_0(t)=\pm((2m)^{\frac{1}{2m}}(z_0^\ast)^{\frac{2m-1}{2m}}+O(t^{\frac{2m-1}{2m}})).$$
Then the parametrization of $\mathcal{C}$ is $t\mapsto (\tilde x(t),
\tilde z(t))$ where
\begin{align}
\tilde
x(t)&=t^{2m-1}x_0(t)=(2m)^{\frac{1}{2m}}(z_0^\ast)^{\frac{2m-1}{2m}}t^{2m-1}+O(t^{\frac{4m^2-1}{2m}})\label{para-x}\\
\tilde z(t)&=t^{2m}z_0(t)=z_0^\ast t^{2m}+O(t^{2m+1}).\label{para-z}
\end{align}
Recall that the relation between $(\tilde x,\tilde z)$ and $(x,h)$
is given by $\tilde x=x/h$ and $\tilde z=1/h$. From (\ref{para-z})
it follows that $1/h\sim z_0^\ast t^{2m}$ at $h=\infty$. Using this
behavior  and (\ref{para-x}) we get $x\sim\pm(2mh)^{\frac{1}{2m}}$
at $h=\infty$. Hence  $F_\pm^{-1}(h)\sim(2mh)^{\frac{1}{2m}}$ and
from \eqref{lf-lg} it follows that
$\textit{l}_F(h)\sim2(2mh)^{\frac{1}{2m}}$.
\end{proof}

\subsection{The Harmonic Balance Method}\label{bhm}

In this section we recall the $N$-th order HBM adapted to  our
setting. Consider the second order differential equation
\begin{equation*}
       \ddot{x}=f(x,\alpha), \quad\alpha\in\R.
\end{equation*}
Suppose that it has a $T$-periodic solution $x(t)$ such that $x(0)=A$ and
$\dot{x}(0)=0.$ This $T$-periodic function $x(t)$ satisfies the
functional equation
\begin{equation*}
\mathcal{F}:=\mathcal{F}(x(t),\ddot{x}(t),\alpha)=\ddot{x}(t)-f(x(t),\alpha)=0.
\end{equation*}

On the other hand, $x(t)$ has the Fourier series:
$$
x(t)= \frac{\tilde{a}_0}2+\sum_{k=1}^{\infty} \left(\tilde{a}_k
\cos(k\omega t)+ \tilde{b}_k\sin(k\omega t)\right),
$$
where  $\omega:=2\pi/T$ is the angular frequency of $x(t)$ and the
coefficients $\tilde{a}_k$ and $\tilde{b}_k$ are the so-called
Fourier coefficients, which are defined as
\[
\tilde{a}_{k}=\frac{2}{T}\int_{0}^{T}x(t)\cos(k\omega
t)\,dt\quad\mbox{and}\quad
\tilde{b}_{k}=\frac{2}{T}\int_{0}^{T}x(t)\sin(k\omega t)\,dt
\quad\mbox{for $k\ge 0$}.
\]
(Although we not write explicitly,
$\tilde{a}_k$, $\tilde{b}_k$, and $\omega$ depend on $\alpha$ and
$A$, that is, $\tilde{a}_k:=\tilde{a}_k(\alpha,A)$,
$\tilde{b}_k:=\tilde{b}_k(\alpha,A)$, and $\omega:=\omega(\alpha,
A)$.)
Hence it is natural to try to approximate the periodic solutions of
the functional equation $\mathcal{F}=0$ by using truncated Fourier
series of order $N$, {\it i.e.} trigonometric polynomials of degree
$N$.

The $N$-th order HBM consists of the following four steps.

1. Consider a trigonometric polynomial
\begin{equation}\label{solu}
x_{_N}(t)=\frac{a_0}2+\sum_{k=1}^{N} \left(a_k \cos(k\omega_N t)+
b_k\sin(k\omega_N t)\right).
\end{equation}

2. Compute the $T$-periodic function
$\mathcal{F}_N:=\mathcal{F}(x_{_N}(t),\ddot{x}_{_N}(t))$, which has
also an associated Fourier series, that is,
\[
\mathcal{F}_N=\frac{\mathcal{A}_0}2+\sum_{k=1}^{\infty}
\left(\mathcal{A}_k \cos(k\omega_N t)+ \mathcal{B}_k\sin(k\omega_N
t)\right),
\]
where $\mathcal{A}_k=\mathcal{A}_k({\mathbf a},{\bf b},\omega)$ and
$\mathcal{B}_k=\mathcal{B}_k({\mathbf a},{\bf b},\omega)$, $k\ge0,$
with ${\mathbf a}=(a_0,a_1,\ldots,a_{_N})$ and ${\bf
b}=(b_1,\ldots,b_{_N})$.

3. Find values ${\bf a}$, ${\bf b}$,
and $\omega$ such that
\begin{equation}\label{sistema}
\mathcal{A}_k({\bf a},{\bf
b},\omega)=0\quad\mbox{and}\quad\mathcal{B}_k({\bf a},{\bf
b},\omega)=0 \quad \mbox{for}\quad 0\le k\le N.
\end{equation}

4. Then the expression \eqref{solu}, with the values of ${\bf a}$,
${\bf b}$, and $\omega$ obtained in point~3, provides candidates to
be
 approximations of the actual periodic solutions of the initial
 differential equation. In particular the values
 $2\pi/\omega$ give approximations of  the periods of the
 corresponding periodic orbits.

\smallskip

We end this short explanation about HBM with several comments:

(a) The above set of equations \eqref{sistema} is a system of
polynomial equations which usually is very difficult to solve. For
this reason in many works, see for instance \cite{Mi,Mic-book2} and
the references therein, only small values of $N$ are considered. We
also remark that in general the coefficients of $x_{_N}(t)$ and
$x_{_{N+1}}(t)$ do not coincide at all. Hence, going from order $N$
to order $N+1$ in the method, implies to compute again all the
coefficients of the Fourier polynomial.

(b) The equations $\mathcal{A}_k({\bf a},{\bf b},\omega)=0$ and
$\mathcal{B}_k({\bf a},{\bf b},\omega)=0$ for $0\le k\le N$ are
equivalent to
\[
\frac{2}{T}\int_0^{T} \mathcal{F}_N\cos (k\omega
t)\,dt=0\quad\mbox{and}\quad\frac{2}{T}\int_0^{T} \mathcal{F}_N\sin
(k\omega t)\,dt=0 \quad \mbox{for}\quad 0\le k\le N.
\]

(c) The linear combination, $ a_k \cos (k \omega t)+b_k \sin (k
\omega t)$, of the harmonics of order $k$, with $k=0,1,\ldots N$,
can be expressed as
$$
a_k \cos (k \omega t)+b_k \sin (k \omega t)= \frac{\bar{c}_k\, e^{i
k \omega t}+ c_k\, e^{-i k \omega t}}{2},
$$
where $c_k=a_k+i b_k$, and $\bar{c}_k$ is the complex conjugated of
$c_k$. Therefore, we can use the HBM with the last notation, because
the truncated Fourier series can be written as
\begin{equation}\label{HBcomp}
       x_N(t)=\sum_{k=1}^{N} \frac{\bar{c}_k\, e^{i k \omega_N t}+
       c_k\, e^{-i k \omega_{N} t}}{2}, \quad k=1,\ldots, N.
\end{equation}

(d) In general, although in many concrete applications HBM seems to
give quite accurate results, it is not proved that the found Fourier
polynomials are approximations of the  actual periodic solutions of
differential equation. Some attempts to prove this relationship can
be seen in \cite{Ga-Ga} and the references therein.

\section{The period function from the analytical point of view}\label{analyt}

In this section we prove our  main results concerning the period
function of systems \eqref{sis} and \eqref{sis2}. For proving
Theorem~\ref{teo 1.1} we will apply Lemma~\ref{lemcent} and
Proposition~\ref{propo10} to determine the existence of a global
center  of \eqref{sis} and the monotonicity of its period function.
To find the Taylor series of $T$ at the origin we will use an old
idea, due to Cherkas(\cite{Che}),  which consists in transforming
\eqref{sis} into an Abel equation. Finally, in the last part of the
proof, that corresponds to the behavior at infinity of $T$, we will
use Theorem~\ref{teoC} and Lemma \ref{lema}. Theorem~\ref{tem 1.3}
follows using similar tools.

\subsection{Proof of Theorem~\ref{teo 1.1}}

System \eqref{sis} is  of the form \eqref{hamil} with
$F(x)=x^2/2+x^{2m}/2m$. Clearly, by Lemma~\ref{lemcent}, the origin
$(0,0)$ is a global center. Moreover, the set $\{(A,0)\in \R^2 \, |
\, A>0\}$ is a transversal section to $\mathcal{P}$. Thus, $T$ can
be expressed as function depending on the parameter $A$.

Some easy computations give that
$$
\begin{array}{l}
F^{\prime}(x)^2-2F(x)F^{\prime
\prime}(x)=-(\frac{m-1}{m})((2m-1)+x^{2m-2})x^{2m}\leq 0.
\end{array}
$$
Therefore,  Proposition~\ref{propo10}.$(ii)$ implies that the period
function $T(A)$ associated to $\mathcal{P}$ is decreasing for all
$m$.

For obtaining the Taylor series of $T$ at $A=0$ we will consider
system \eqref{sis} in polar coordinates and initial condition
$(A,0)$, that is,
\begin{equation}\label{polaressis}
\left\{\begin{array}{l}
\dot R=\sin(\theta) \cos ^{2m-1}(\theta)R^{2m-1}\\
\dot \theta=1+\cos ^{2m}(\theta)R^{2m-2},
\end{array}\right. \quad R(0)=A,\quad \theta(0)=0,
\end{equation}
which is equivalent to the differential equation
$$
\frac{dR}{d\theta}=\frac{\sin(\theta) \cos
^{2m-1}(\theta)R^{2m-1}}{1+\cos ^{2m}(\theta)R^{2m-2}}, \quad R(0)=A.
$$
By applying the Cherkas transformation \cite{Che}:
$r=r(R;\theta)=\frac{R^{2m-2}}{1+\cos ^{2m}(\theta)R^{2m-2}}$ to the
previous equation, we obtain the Abel differential equation
\begin{equation}\label{eqAbel}
\frac{dr}{d \theta}=P(\theta)r^3+Q(\theta)r^2, \quad
r(A;0)=\frac{A^{2m-2}}{1+A^{ 2m-2}},
\end{equation}
where $P(\theta)=(2-2m)\sin(\theta)\cos ^{4m-1}(\theta)$ and
$Q(\theta)=2(2m-1)\sin(\theta)\cos ^{2m-1}(\theta).$ Near the
solution $r=0$, the solutions of this Abel equation can be written
as the power series
\begin{equation}\label{eqSer}
r(A;0)=\frac{A^{2m-2}}{1+A^{
2m-2}}+\sum_{i=2}^{\infty}u_i(\theta)\left(\frac{A^{2m-2}}{1+A^{2m-2}}
\right)^i
\end{equation}
for some functions $u_i(\theta)$ such that $u_i(0)=0$ which can be
computed solving recursively linear differential equations obtained
by replacing \eqref{eqSer} in \eqref{eqAbel}. For instance,
$$
u_2(\theta)=\int_{0}^{\theta} Q(\psi)d\psi \qquad {\rm and} \qquad
u_3(\theta)=\int_{0}^{\theta} (P(\psi)+2Q(\psi)u_2(\psi))d\psi.
$$
From the expression of $\dot{\theta}$ in \eqref{polaressis} and
using variables $(r,\theta)$ again, we obtain
$$
\begin{array}{rl}
\displaystyle
T(A)&\displaystyle=\int_0^{2\pi}\frac{d\theta}{1+\cos
^{2m}(\theta)R^{2m-2}}
\displaystyle=\int_0^{2\pi}(1-\cos ^{2m}(\theta)\,r)d\theta=\\
&\displaystyle=2\pi-\int_0^{2\pi}\cos^{2m}(\theta)\left(\frac{A^{2m-2}}{
1+A^{2m-2}}+\sum_{i=
2}^{\infty}u_i(\theta)\left(\frac{A^{2m-2}}{1+A^{2m-2}}\right)^i\right)d\theta.
\end{array}
$$
Then, we have
$$
T(A)=2\pi-\sum_{k\geq1}\mathcal{S}
_k\left(\frac{A^{2m-2}}{1+A^{2m-2}}\right)^k,
$$
with
$$
\mathcal{S}_1=\int_0^{2\pi}\cos ^{2m}(\theta)d\theta \quad {\rm
and} \quad \mathcal{S}_k=\int_0^{2\pi}\cos ^{2m}(\theta)u_k(\theta)
d\theta \quad {\rm for} \quad k\geq 2.
$$
It is easy to see that for $|A|<1,$
$$\frac{A^{2m-2}}{1+A^{2m-2}}=A^{2m-2}-A^{4m-4}+O(A^{6m-6}),\,
\Big{(}\frac{A^{2m-2}}{1+A^{2m-2}}\Big{)}^2=A^{4m-4}+O(A^{6m-6}).$$
Thus,
$$
T(A)=2\pi-
\mathcal{S}_1A^{2m-2}-(\mathcal{S}_2-\mathcal{S}_1)A^{4m-4}
-O(A^{6m-6}).
$$
Easy computations show that
$$\mathcal{S}_1=2\pi\frac{(2m-1)!!}{(2m)!!},\qquad
\mathcal{S}_2=2\pi\left(\frac{2m-1}{m}\right)\left(\frac{(2m-1)!!}{(2m)!!}-\frac
{(4m-1)!!}{(4m)!!}\right),$$ where, given $n\in\N^+$, $n!!$ is
defined recurrently as $n!!=n\times(n-2)!!$ with 1!!=1 and  2!!=2.
Hence, introducing ${S}(m)=\mathcal{S}_2-\mathcal{S}_1$ we obtain
\eqref{tay}, as we wanted to prove.

Finally, for studying the behavior of $T$ at infinity we will apply
Theorem~\ref{teoC}. By  Lemma~\ref{lemlf1}, we have that at
$h=\infty,$ $\textit{l}_F(h)\sim 2(2mh)^{\frac{1}{2m}}$. Then
$$
\textit{l}^{\,\prime}_F(h)\sim\frac{(2m)^{\frac{1}{2m}}}{m}h^{-\frac{2m-1}{2m}}.
$$

If we denote by $\widetilde{T}(h)$ the period function of
(\ref{sis}) in terms of $h$, then, from Theorem~\ref{teoC}, it
follows that $\widetilde{T}(h)$ at $h=\infty$ satisfies
\begin{equation}
\label{PFInfgenh} \widetilde{T}(h) \sim
B\left(\frac{1}{2m},\frac{1}{2}\right)2^{\frac{m+1}{2m}}m^{-\frac{2m-1}{2m}}h^{-\frac{m-1}{2m}}.
\end{equation}
Now, using that  $h=A^2/2+A^{2m}/2m$, we get
$$
h^{-\frac{m-1}{2m}}=\left(
\frac{A^2}{2}+\frac{A^{2m}}{2m}\right)^{-\frac{m-1}{2m}}=
A^{-(m-1)}\left(
\frac{1}{2A^{2m-2}}+\frac{1}{2m}\right)^{-\frac{m-1}{2m}}
$$
and we have
$$
\lim_{A\to \infty}\frac{A^{-(m-1)}\left(
\frac{1}{2A^{2m-2}}+\frac{1}{2m}\right)^{-\frac{m-1}{2m}}-(2m)^{\frac{m-1}{2m}}A^{-(m-1)}}{A^{-(m-1)}}=0.
$$
Hence  $T(A)=\widetilde{T}(A^2/2+A^{2m}/2m)$, and from previous
equation and \eqref{PFInfgenh} we obtain
 $$T(A)\sim
B\left(\frac{1}{2m},\frac{1}{2}\right)2^{\frac{m+1}{2m}}m^{-\frac{2m-1}{2m}}
(2m)^{\frac{m-1}{2m}}A^{-(m-1)},
$$
which after a simplification reduces to \eqref{PFInfgen}.

\subsection{Proof of Theorem~\ref{tem 1.3}}
By using the
transformation $u=x/k$, $v=yk^{m-1}$, and the rescaling of time
$\tau=-t/k^m$,  system (\ref{sis2}) becomes
\begin{equation}\label{sis2p}
\left\{\begin{array}{lll}
\dot{x}=-y,\\
\dot{y}=\displaystyle\frac{x}{(x^{2}+1)^{m}},\qquad m\in[1,\infty),
\end{array}\right.
\end{equation}
where we have reverted to the original notation $(x,y)$ and $t$.

The associated Hamiltonian function to \eqref{sis2p} is
$H(x,y)=\frac{y^{2}}{2}+F(x)$ with
$$
F(x)=\left\{\begin{array}{ll} -\frac{1}{2}\ln(x^2+1),& \mbox{if
$m=1$,}\\ \\ -\frac{1}{2(m-1)(x^{2}+1)^{m-1}}+\frac{1}{2(m-1)}, &
\mbox{if $m>1$.}\end{array}\right.
$$

It is clear that for all $m$ the function $F$ is smooth at the
origin and has   a non-degenerate minimum. Thus, from
Lemma~\ref{lemcent}, system \eqref{sis2p} has a center at the origin
with some period annulus $\mathcal{P}$.

From a straightforward computation we get
$$
F^{\prime}(x)^2-2F(x)F^{\prime \prime}(x)=\left\{\begin{array}{ll}
\frac{x^2+(x^2-1)\ln(x^2+1)}{(x^2+1)^2}, & \mbox{if $m=1$,}\\
\\\frac{1-mx^{2}+((2m-1)x^{2}-1)(x^{2}+1)^{m-1}}{(m-1)(x^{2}+1)^{2m}},
& \mbox{if $m>1$.}\end{array}\right.
$$

To prove that the period function $T$ associated to $\mathcal{P}$ is
increasing we will apply Proposition~\ref{propo10}.($i$). Hence we
need only to show that $F^{\prime}(x)^2-2F(x)F^{\prime
\prime}(x)\geq 0$. For $m=1$ it is clear. For $m>1$ the denominator
of $F^{\prime}(x)^2-2F(x)F^{\prime \prime}(x)$ is positive, then
remains to prove that its numerator is positive.

By taking  $w=x^2+1$, the numerator of
$F^{\prime}(x)^2-2F(x)F^{\prime \prime}(x)$ with $m>1$ is
$(2m-1)w^m-2mw^{m-1}-mw+m+1$ or equivalently,
$$
(w-1)^2\left((2m-1)w^{m-2}+(2m-2)w^{m-3}+\ldots+m+1\right),
$$
which is clearly  positive.

To finish the proof, we will discuss about the globality of the
center. For $m=1$ the $(0,0)$ is a global minimum of $H$. Thus,
(\ref{sis2p}) and therefore (\ref{sis2}) have a global center at the
origin. For $m>1$ the level curve
$$\mathcal{C}_{\frac{1}{2m-2}}=\left\{\frac{1}{2(m-1)(x^{2}+1)^{m-1}}-\frac{y^{2}}{2}=0\right\}$$
has two disjoin components. Indeed, it is formed by the graphics of
the functions
$$y=\pm\frac{1}{\sqrt{(m-1)(x^{2}+1)^{m-1}}},
$$
which are well-defined for all $x\in\R$ because $m>1$.  This implies
that the center at the origin of \eqref{sis2p} is bounded by
$\mathcal{C}_{\frac{1}{2m-2}}$ and  therefore it is not global. The
same happens with \eqref{sis2}.

\section{The period function from the point of view of HBM}\label{hbm proof} In this
section we prove Propositions \ref{gen-Duf-HB} and  \ref{genMI-BH}.

\subsection{Proof of Proposition \ref{gen-Duf-HB}}
System \eqref{sis} is equivalent to the  second order differential
equation $\ddot{x}+x+x^{2m-1}=0$ with initial conditions $x(0)=A$,
$\dot x(0)=0$. For applying  HBM we consider the functional equation
\begin{equation}\label{fo HBM}
\mathcal{F}(x(t),\ddot{x}(t))=\ddot{x}(t)+x(t)+x(t)^{2m-1}=0.
\end{equation}
By symmetry, for applying the 1st order HBM  we can look for a
solution of the form $x(t)=a_1\cos(\omega_1 t)$. We substitute it in
\eqref{fo HBM}. By using that
$$
\cos^{2m-1}(\omega_1 t)=\frac{1}{2^{2m-2}}\sum_{k=0}^{m-1}{2m-1
\choose k}\cos ((2m-2k-1)\omega_1 t)
$$
and reordering terms we have that the vanishing of the coefficient
of $\cos(\omega_1 t)$ in $\mathcal{F}_1(x(t),\ddot{x}(t))$ implies
$$
2^{2m-2}(\omega_1^2-1)-\frac{(2m-1)!}{(m-1)!\,m!} a_1^{2m-2}=0.
$$
From the initial conditions  we have $a_1=A$, whence
$$
\omega_1=\frac{1}{2^{m-1}}\sqrt{\frac{(2m-1)!}{(m-1)!m!}A^{2m-2}+2^{2m-2}}.
$$
Therefore, the first approximation $T_1(A)$ to $T(A)$ of
system~\eqref{sis} is
\begin{equation}\label{fper1}
T_1(A)=\frac{2\pi
}{\frac{1}{2^{m-1}}\sqrt{\frac{(2m-1)!}{(m-1)!m!}A^{2m-2}+2^{2m-2}}}.
\end{equation}
Easy computations shows that the Taylor series of $T_1$ at $A=0$ is
$$
T_1(A)=2\pi\left(1-
\frac{(2m)!}{(m!)^2\,2^{2m}}\,A^{2m-2}+\left(\frac{(2m)!}{(m!)^2\,2^{2m}}\right)^2\,A^{4m-4}+O(A^{6m-6})\right).
$$
By using the identities $(2m)!/(2^m\,m!)=(2m-1)!!$ and
$2^m\,m!=(2m)!!$  we have the expression of the statement.

For studying the behavior at infinity we can write $T_1(A)$ as
$$
T_1(A)=2^m\pi
A^{-m+1}\left(\frac{(2m-1)!}{(m-1)!\,m!}+\frac{2^{2m-2}}{A^{2m-2}}\right)^{-1/2},
$$
thus,
$$
\lim_{A\to \infty}\frac{2^m\pi
A^{-m+1}\left(\frac{(2m-1)!}{(m-1)!m!}+\frac{2^{2m-2}}{A^{2m-2}}\right)^{-1/2}-2^m\pi
A^{-m+1}\left(\frac{(2m-1)!}{(m-1)!m!}\right)^{-1/2}}{A^{-m+1}}=0.
$$
Hence $T_1(A)$ at infinity satisfies \eqref{PF1ogeninf}.

\subsection{Proof of Proposition \ref{genMI-BH}}
System (\ref{sis2}) is equivalent  to the  second order differential
equation
\begin{equation}\label{gral}
(x^2+k^2)^m \ddot{x}+x=0,
\end{equation}
with initial conditions $x(0)=A$, $\dot{x}(0)=0$. For simplicity in
the computations, we consider the complex form, given in
\eqref{HBcomp}, of the first-order HBM
\begin{equation}\label{m-or1c}
x_1(t)=\frac{1}{2}\left (\bar{c} e^{i\omega_1 t}+ce^{-i\omega_1
t}\right),
\end{equation}
where $c=a+bi$. By using the binomial expression
$$
(x^2+k^2)^m=\sum_{j=0}^{m}{m\choose j} x^{2j}k^{2(m-j)},
$$
and by replacing \eqref{m-or1c} in \eqref{gral}, after some
computations we get
$$
\begin{array}{l}
\displaystyle -\omega_1^2\sum_{j=0}^{m}{m\choose j}
\left(\frac{1}{2}\right)^{2j+1}k^{2(m-j)}\sum_{l=0}^{2j+1}{{2j+1}\choose
l} (c^l\bar{c}^{\,2j-l+1})(e^{i\omega_1 t})^{2j-2l+1}\\\\
\phantom{aaaaa}\displaystyle+\frac{\left (\bar{c} e^{i\omega_1
t}+ce^{-i\omega_1 t}\right )}{2}=0.
\end{array}
$$

We are concerned only with the first-order harmonics, i.e. $j=l$ or
$l=j+1$ in the above equation
$$
\begin{array}{l}
\displaystyle-\frac{1}{2}\left(\omega_1^2\sum_{j=0}^{m}\left(\frac{1}{2}\right)^{2j}
{m\choose j} {{2j+1}\choose j}k^{2(m-j)}
(c\bar{c})^{j}-1\right)\bar{c}e^{i\omega_1 t}\\\\
\phantom{aaaaa}\displaystyle-\frac{1}{2}\left(\omega_1^2\sum_{j=0}^{m}
\left(\frac{1}{2}\right)^{2j}{m\choose j}{{2j+1}\choose
{j+1}}k^{2(m-j)} (c\bar{c})^{j}-1\right)ce^{-i\omega_1 t}+HOH=0.
\end{array}
$$
Since ${{2j+1}\choose {j}} ={{2j+1}\choose {j+1}} $, the previous
equation can be written as $$
-\left(\omega_1^2\sum_{j=0}^{m}\left(\frac{1}{2}\right)^{2j}
{m\choose j} {{2j+1}\choose j} k^{2(m-j)}
(c\bar{c})^{j}-1\right)\left(\frac{\bar{c}e^{i\omega_1
t}+ce^{-i\omega_1 t}}{2}\right)=0,
$$
whence
$$
\omega_1=\frac{1}{\sqrt{\sum_{j=0}^{m}\left(\frac{1}{2}\right)^{2j}{m\choose
j} {{2j+1}\choose j} k^{2(m-j)}(c\bar{c})^{j}}}.
$$
By the initial conditions we have $a_1=A$ and $b_1=0$ then $c\bar
c=A^2$. Therefore, the approximation $T_1(A)$ of $T(A)$ associated
to system (\ref{sis2}) is
$$
T_1(A)=2\pi\sqrt{\sum_{j=0}^{m}\left(\frac{1}{2}\right)^{2j}{m\choose
j} {{2j+1}\choose j} k^{2(m-j)}A^{2j}}.
$$

\section{The Duffing-harmonic oscillator}\label{Duff}
This section is devoted to the study of the Duffing-harmonic
oscillator. We compare the  approximations  $T_N(A), N=1,2,3, $
given by the $N$-th order HBM with the exact period function $T(A)$
of the system
\begin{equation}\label{sDuf}
\left\{\begin{array}{l}
\dot{x}=-y\\
\dot{y}=x+x^3,
\end{array}\right.
\end{equation}
with initial conditions $x(0)=A$, $y(0)=0$, both near the origin and
at infinity. Our results extend those of \cite{Mic-book2}, where
only the cases $N=1,2$ are studied and where the analytic
comparaison is restricted to a neighborhood of the origin.
\begin{remark}
Some papers $($for instance \cite{Fo,Zhang}$)$ consider the
Duffing-harmonic oscillator $\dot{x}=-y$, $\dot{y}=x+\epsilon x^3$,
$\epsilon\ne0$, however, it is not difficult to see that by applying
the transformation $x=\epsilon^{-1/2}u$, $y=\epsilon^{-1/2}v$, this
system becomes \eqref{sDuf}.
\end{remark}
As in \cite{Mic-book2}, we compute the period function $T(A)$ of
\eqref{sDuf} via elliptic functions. Let us remember the {\bf K}
complete elliptic integral of the first kind see \cite[pp. 590]{AS}
$$
{\bf K}(k)=\int_0^1\frac{dz}{\sqrt{(1-z^2)(1-k\,z^2)}},
$$
whose Taylor expansion at $k=0,$ for $| k|<1$ is
\begin{equation}\label{exIne}
{\bf K}(k)=
\frac{1}{2}\pi\left[1+\left(\frac{1}{2}\right)^2k+\left(\frac{1\cdot3}{2\cdot
4}\right)^2k^2+\left(\frac{1\cdot3\cdot5}{2\cdot
4\cdot6}\right)^2k^3+\cdots \right].
\end{equation}
\begin{lemma}\label{lDuf}
The period function $T(A)$ associated to the system~\eqref{sDuf} is
given by
\begin{equation}\label{eee}
T(A)=\frac{4}{\sqrt{1+\frac{1}{2}A^2}}\,{\bf
K}\left({\frac{-A^2}{2+A^2}}\right).
\end{equation}
 Moreover, its Taylor series at $A=0$ is
\begin{equation}\label{stx3}
T(A)=2\pi-\frac{3}{4}\pi A^2+\frac{57}{128}\pi
A^4-\frac{315}{1024}\pi A^6+\frac{30345}{131072}\pi A^8+O(A^{10}).
\end{equation}
and its behavior at infinity is
\begin{equation}\label{InfDuf}
T(A)\sim
B\left(\frac{1}{4},\frac{1}{2}\right)\frac{\sqrt{2}}{A}\approx
\frac{7.4163}{A}.
\end{equation}
\end{lemma}

\begin{proposition}\label{proDuf} Let $T_N(A), N=1,2,3,$ be the approximations of the period
function for system~\eqref{sDuf} obtained applying the $N$-th order
HBM. Then:
\begin{itemize}
\item[$(i)$]
The first approximation is
$$T_1(A)=\frac{4\pi}{\sqrt{3A^2+4}}.$$
Its Taylor series at $A=0$ is
\begin{equation}
\label{PF1ox3} T_1(A)=2\pi-\frac{3}{4}\pi A^2+\frac{27}{64}\pi
A^4+O(A^6),
\end{equation}
and its behavior at  infinity
\begin{equation}
\label{PF1ox3inf}
T_1(A)\sim\frac{4\pi}{\sqrt{3}A}\approx\frac{7.2551}{A}.
\end{equation}
\item[$(ii)$] The second approximation is
$$T_2(A)=\frac{2\pi}{\omega_2(A)},$$
where $\omega_2(A):=\omega_2$ is the real positive solution to the
equation
$$
1058\,\omega_2^6-3(219A^2+322)\omega_2^4-\frac{9}{4}(21A^4+80A^2+40)\omega_2^2
-\frac{27}{64}A^2(7A^2+8)^2-2=0.
$$
Moreover, its Taylor series at $A=0$ is
\begin{equation}\label{PF2ox3}
T_2(A)=2\pi-\frac{3}{4}\pi A^2+\frac{57}{128}\pi
A^4-\frac{633}{2048}\pi A^6+O(A^8),
\end{equation}
and its behavior at infinity is given by
\begin{equation}
\label{PF2ox3inf}
T_2(A)\sim\frac{\bar{\Delta}}{A}\approx\frac{7.4018}{A},
\end{equation}
where
\begin{align*}
\bar{\Delta}=\, \frac{92\sqrt{2}\pi \Delta}{
\sqrt{1033992+876\Delta+\Delta^{2}}},\quad
\Delta=(1763014086+71386434\sqrt{393})^{1/3}.
\end{align*}

\item[$(iii)$] The third approximation  is given implicitly as
one of the branches of an algebraic curve  $h(A^2,T^2)=0$ that has
degree 11 with respect to $A^2$ and $T^2$ and total degree 44. In
particular, at $A=0,$
\[
T_3(A)=2\pi-\frac{3}{4}\pi A^2+\frac{57}{128}\pi
A^4-\frac{315}{1024}\pi A^6+\frac{30339}{131072}\pi A^8+O(A^{10})
\]
and, at infinity,
\begin{equation}
\label{PF2ox3inf} T_2(A)\sim\frac{\delta}{A}\approx\frac{7.4156}{A},
\end{equation}
where $\delta$ is the positive real root of an even  polynomial of
degree 22.
\end{itemize}
\end{proposition}

Notice that by Lemma \ref{lDuf} and Proposition \ref{proDuf} it
holds that
\[
T(A)-T_N(A)=O(A^{2N+4}),\quad N=1,2,3,
\]
result that evidences that, at least locally and for these values of
$N$, the $N$-th order HBM improves when $N$ increases. Moreover, the
dominant terms at $A=\infty$ of $T_N(A)$ also improve when $N$
increases.

In Figure~\ref{error} it is shown the absolute error between the
exact period function $T(A)$ and first and second approximation by
using HBM.

\begin{figure}[h]
\centering\epsfig{file=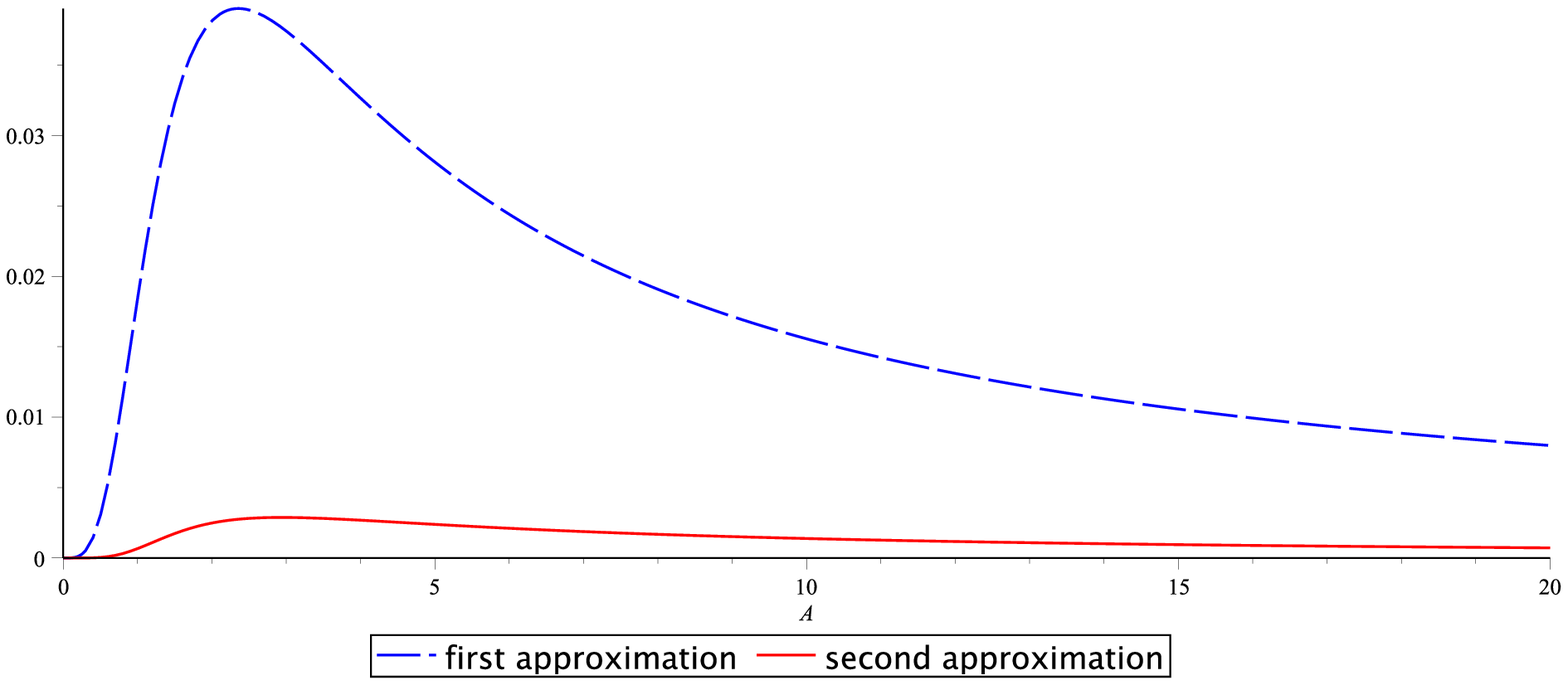,width=280pt,height=160pt}
\caption{}\label{error}
\end{figure}

\begin{proof}[Proof of Lemma~\ref{lDuf}]
The Hamiltonian function associated to \eqref{sDuf} is
$H(x,y)=y^2/2+x^2/2+x^4/4.$ The energy level is
$H(x,y)=A^2/2+A^4/4:=h$. The expression of the period function is
$$
T(A)=4\int_0^A\frac{dx}{\sqrt{2h-x^2-\frac{1}{2}x^4}}.
$$
Making the change of variable $z=x/A$, we can write the above
expression as
$$
T(A)=\frac{4}{\sqrt{1+\frac{1}{2}A^2}}\int_0^1\frac{
dz}{\sqrt{(1-z^2)\left(1 +\frac{A^2}{2+A^2} z^2\right)}},
$$
which gives the expression~\eqref{eee} of $T$ that we wanted to
prove. By using \eqref{exIne} and \eqref{eee}, straightforward
computations yield to~\eqref{stx3}. The behavior at  infinity of
$T(A)$ is a direct consequence of Theorem~\ref{teo 1.1} with $m=2$.

\end{proof}

\begin{proof}[Proof of Proposition~\ref{proDuf}]
Notice that result $(i)$ corresponds to the particular case $m=2$ in
Proposition~\ref{gen-Duf-HB}. In this case, expression~\eqref{fper1}
gives
$$
T_1(A)=\frac{4\pi}{\sqrt{3A^2+4}}.
$$
Straightforward computations show that its  Taylor series at $A=0$
is \eqref{PF1ox3}. Moreover, writing $T_1(A)$ as
$$
 T_1(A)=\frac{4\pi}{A\sqrt{3+\frac{4}{A^2}}},
$$
it  is clear that $\frac{4\pi}{\sqrt{3}A}$ is its dominant term at
infinity.

\smallskip
$(ii)$ System~\eqref{sDuf} is equivalent  to the  second order
differential equation
\begin{equation}\label{edod}
\ddot{x}+x+x^{3}=0, \end{equation} with initial conditions $x(0)=A$,
$\dot{x}(0)=0$.

For applying second-order HBM we look for an solution
of~\eqref{edod} of the form $x_2(t)=a_1\cos(\omega_2
t)+a_3\cos(3\omega_2 t)$ the above differential equation. The
vanishing of the coefficients of $\cos(\omega_2 t)$ and
$\cos(3\omega_2 t)$ in the Fourier series of $\mathcal{F}_2$
provides the non-linear system
\begin{align*}
 -4\omega_2^2+3a_1^{2}+3a_1a_3+6a_3^2+4&= 0,
\\ -9a_3\omega_2^2+\frac{1}{4}a_1^{3}+\frac{3}{2}a_1^2a_3+a_3+\frac{3}{4}a_3^3 &=0.
\end{align*}
From the initial conditions we have $a_1=A-a_3$. Hence the above
system becomes
\begin{align*}
-4\omega_2^2+6a_3^2-3a_3A+3A^2+4&=0, \\
-9\omega_2^2a_3+2a_3^3-\frac{9}{4}a_3^2A+\frac{3}{4}a_3A^2+a_3+\frac{1}{4}A^3&=0.
\end{align*}
Doing the resultant of these equations with respect to $a_3$, we
obtain the polynomial
$$
1058\omega_2^6-3(219A^2+322)\omega_2^4-\frac{9}{4}(21A^4+80A^2+40)\omega_2^2-
\frac{1323}{64}A^6-\frac{189}{4}A^4-27A^2-2.
$$
Thus, $\omega_2$ is the unique real positive root of the above
polynomial, that is,
\begin{align*}
\omega_2=&\frac{\sqrt{2}}{92}\left(\frac{2166784+3272256\,A^2+1033992\,A^4+(1288+876
A^2)\,R^{1/3}+\,R^{2/3}}{R^{1/3}}\right)^{1/2}
\end{align*}
where
\begin{align*}
R&=3189506048+7956430848\,A^2+6507324864\,A^4+1763014086\,A^6
\\&\,\,\quad +3174\,(320+357A^2)\,A\,S,\\
S&=(4521984+9925632\,A^2+6899904\,A^4+1559817\,A^6)^{1/2}.
\end{align*}
Therefore, the second approximation $T_2(A)$ to $T(A)$ of
\eqref{sDuf} is $ T_2(A)=2\pi/\omega_2$, and it is not difficult to
see that its Taylor series at $A=0$ is \eqref{PF2ox3}.

For studying the behavior of $T_2$ at infinity we rewrite $\omega_2$
as
$$
\omega_2=\frac{\sqrt{2}\,A}{92\,\bar{R}^{1/6}}\left(\frac{2166784}{A^4}+\frac{3272256}{A^2}
+1033992+\left(\frac{1288}{A^4}+\frac{876
}{A^2}\right)\,\bar{R}^{1/3}+\,\bar{R}^{2/3}\right)^{1/2}
$$
where
\begin{align*}
\bar{R}&=\frac{3189506048}{A^6}+\frac{7956430848}{A^4}+\frac{6507324864}{A^2}
+1763014086\\&\,\,\quad
+\frac{1015680}{A^5}\bar{S}+1133118\bar{S},\\
\bar{S}&=\left(\frac{4521984}{A^6}+\frac{9925632}{A^4}+
\frac{6899904}{A^2}+1559817\right)^{1/2}.
\end{align*}
From the previous expressions we have
$$
\lim_{A \to \infty} \bar{S}=63\sqrt{393}, \qquad \lim_{A \to
\infty}\bar{R}=1763014086+71386434\sqrt{393}.
$$
Thus,
\begin{align*}
\displaystyle \lim_{A \to \infty}\frac{2\pi A}{\omega_2}=\,
&\displaystyle \frac{92\sqrt{2}\pi \Delta}{
\sqrt{1033992+876\Delta+\Delta^{2}}},
\end{align*}
where $\Delta=(1763014086+71386434\sqrt{393})^{1/3}$.

Hence,
$$
\lim_{A \to \infty}\frac{T_2(A)-\bar{\Delta}A^{-1}}{A^{-1}}=0,
$$
where
\begin{align*}
\bar{\Delta}=\, &\displaystyle \frac{92\sqrt{2}\pi \Delta}{
\sqrt{1033992+876\Delta+\Delta^{2}}}.
\end{align*}
Therefore, we have proved $(ii)$.

\smallskip
$(iii)$  When $N=3$ we  look for a solution of \eqref{edod} of the
form $x_2(t)=a_1\cos(\omega_3 t)+a_3\cos(3\omega_3
t)+a_5\cos(5\omega_3 t)$. Using the initial conditions we get that
$a_1=A-a_3-a_5$. Afterwards, imposing that the first three
significative harmonics vanish, we obtain the system of three
equations:

\begin{align*} P=& A-{\omega_3}^{2}A+\frac34 {A}^{3}+ \left( {\omega_3}^{2}-\frac32 {A}^{2}-1
\right) a_{{3}} + \left( {\omega_3}^{2}-\frac94 {A}^{2}-1 \right)
a_{{5}}+\frac92 a_{{3}}a_{{5}}A\\&+\frac94 A{a_{{3}}}^{2}+{\frac
{15}{4}} {a_{{5}}}^{2}A-\frac94 {a_{{5}}}^{3}-
3 {a_{{3}}}^{2}a_{{5}}-\frac92 a_{{3}}{a_{{5}}}^{2}-\frac32 {a_{{3}}}^{3}=0,\\
Q=& \frac14 {A}^{3}+ \left( 1+\frac34 {A}^{2}-9 {\omega_3}^{2}
\right) a_{{3}}-\frac32 a_ {{3}}a_{{5}}A-\frac34
{a_{{5}}}^{2}A-\frac94 A{a_{{3}}}^{2}\\&+\frac32 {a_{{3}}}^{
2}a_{{5}}+\frac94 a_{{3}}{a_{{5}}}^{2}+2 {a_{{3}}}^{3}+\frac12
{a_{{5}}}^{3
}=0,\\
R=& \frac34 {A}^{2}a_{{3}}+ \left( -25 {\omega_3}^{2}+\frac32
{A}^{2}+1 \right) a_{{5} }-3 {a_{{5}}}^{2}A-\frac92
a_{{3}}a_{{5}}A-\frac34 A{a_{{3}}}^{2}\\&+{\frac { 15}{4}}
{a_{{3}}}^{2}a_{{5}}+\frac94 {a_{{5}}}^{3}+{\frac {15}{4}} a_{{
3}}{a_{{5}}}^{2}=0.
\end{align*}

Since all the equations  are polynomial, the searching of its
solutions can be done by using successive resultants, see for
instance \cite{St}. We compute the following polynomials
\[
PQ:=\frac{\operatorname{Res}(P,Q,a_3)}{A-a_5},\quad
QR:=\operatorname{Res}(Q,R,a_3),\] and finally
\[
PQR:=\frac{\operatorname{Res}(PQ,QR,a_5)}{3A^2+4-36\omega_3^2}.
\]
This last expression is a polynomial with rational coefficients that
only depends on $A$ and $\omega_3$ and has total degree 70.
Fortunately, it factorizes as
$PQR(A,\omega_3)=f(A,\omega_3)g(A,\omega_3),$ with factors of
respective degrees 22 and 48. Although both factors could give
solutions of our system we continue our study only with the factor
$f$. It is clear that if we consider the following numerator
\[
h(A^2,T^2):=\operatorname{Num}\Big(f\big(A,\frac{2\pi}T\big)\Big),
\]
we have an algebraic curve $h(A^2,T^2)=0$ that gives a restriction
that has to be satisfied in order to have a solution of our initial
system. This function  is precisely the one that appears in the
statement of the proposition.

Once we have this explicit algebraic curve it is not difficult to
obtain the other results of the statement. So, to obtain the local
behavior near the origin we consider $T_3(A)=\sum_k ^m t_{2k}
A^{2k}$ and we impose that $h(A^2,(T_3(A))^2)\equiv0$, obtaining
easily the first values $t_{2k}.$ Similarly, for $A$ big enough, we
impose that $T_3(A)\sim\delta/A$ obtaining the value of $\delta.$
\end{proof}

\section{Non-monotonous period function}\label{final}

In this section we study the  family of systems~\eqref{sisCri} whose
period function has a critical period (a maximum of the period
function) and we show that the HBM also captures this behavior.

It is not difficult to establish the existence of values of
$k\gtrsim-2$ for which the period function is not monotonous. It
holds that, for all $k$,
\begin{equation}\label{limits}
\lim_{A\to0}T(A)=2\pi\quad\mbox{and}\quad \lim_{A\to\infty}T(A)=0.
\end{equation}

We remark that when $k\le-2$ the center is no more global but there
is also a neighborhood of infinity full of periodic orbits. When
$k=-2$, the system has also  the critical points $(\pm 1,0)$ and all
the orbits of the potential system are closed, except the
heteroclinic ones joining these two points. Hence,  for $k=-2$ and
from the continuity of the flow of \eqref{sisCri} with respect to
initial conditions, it follows that the periodic orbits close to
these heteroclinic orbits have periods arbitrarily high; thus, the
period of nearby periodic orbits, for $k>-2$ with $k+2$ small
enough, is also arbitrarily high due to the continuity of the flow
of \eqref{sisCri} with respect to parameters. Therefore, from this
property and \eqref{limits} it follows that $T(A)$ is not
monotonous.

 The proof that $T(A)$ has only one maximum is much more difficult and
 indeed was the main objective of \cite{MaVi}. In that
 paper the authors proved this fact showing first that
$T(h)$, where $h$ is the energy level of the Hamiltonian associated
with~\eqref{sisCri}, satisfies a Picard-Fuchs equation. As a
consequence,  the function $x(h)=T'(h)/T(h)$ satisfies a Riccati
equation. Finally, they study the flow of this equation for showing
that $x(h)$ vanish at most at a single point.

\begin{proof}[Proof of Proposition \ref{profinal}]
For applying first-order HBM we write the family \eqref{sisCri} as
the second order differential equation
$$
\ddot{x}+x+kx^3+x^5=0.
$$
We look for a solution of the form $x_1(t)=a_1\cos(\omega_1 t)$. The
vanishing of the coefficient of $\cos(\omega_1 t)$ in
$\mathcal{F}_1$, and the initial conditions $x_1(0)=A> 0$,
$\dot{x}_1(0)=0$, provides the algebraic equation
$$
16+12k A^2+10 A^4-16 \omega_1^2 =0.
$$

Solving for $\omega_1$ we obtain
$$
\omega_1(A)=\frac{1}{4}\sqrt{16+12kA^2+10A^4}.
$$
Then, the first approximation $T_1(A)$ to $T(A)$ is
$$
T_1(A)=\frac{8\pi}{\sqrt{16+12kA^2+10A^4}},
$$
which is well defined for all $A\in\R$ only for
$k\in(-2\sqrt{10}/3,\infty)$. It is clear that if $k\geq 0$, then
$T_1(A)$ is decreasing, which proves $(i)$. Moreover
$$
T'_1(A)={\frac {-16\pi \,A \left( 3\,k+5\,{A}^{2} \right) }{
\left(8+ 6\,k{A}^ {2}+5\,{A}^{4} \right) \sqrt
{16+12\,k{A}^{2}+10\,{A}^{4}}}}.
$$
Hence, $T_1(A)$   has a non-zero critical point only when $k\in
(-2\sqrt{10}/3,0)$, and it is $A=\sqrt{-3k}/\sqrt{5}$. Moreover, it
is easy to see that such critical point is a maximum.

The proof of items (ii) and (iii) is straightforward.
\end{proof}

\section{General potential system}\label{secPoten}
In this section we consider the smooth potential system

\begin{equation}\label{potgral}
\left\{\begin{array}{l}
\dot{x}=-y,\\
\dot{y}=x+\displaystyle\sum_{i=2}^\infty k_ix^i.
\end{array}\right.
\end{equation}
Since its Hamiltonian function  has a non degenerated minimum at the
origin, it has a period annulus surrounding the origin. Thus, we
have a period function $T(A)$ associated to this period annulus. The
behavior near the origin of $T(A)$ is given in the following result.

\begin{proposition} The period function $T(A)$ of the system \eqref{potgral}  at $A=0$
  is
\begin{align*}
T(A)=&2\,\pi + \left(\frac{5}{6}\,{k_2}^{2}-\frac{3}{4}\,{k_3}
\right) \pi {A}^{2}+ \left(\frac{5}{9}\,{k_2}^{3}-\frac{1}{2}\,{
k_2}\,{k_3}
 \right)\pi{A}^{3}\\&+\left(\frac{385}{288}k_2^4-\frac{275}{96}k_2^2k_3+\frac{7}{4}k_2k_4
 +\frac{57}{128}k_3^2-\frac{5}{8}k_5\right)\pi A^4 +O \left( {A}^{5}
 \right).
\end{align*}
\end{proposition}

The proof of this proposition follows by using standard methods in
the local study of the period function  \cite{ChiJa, G-G-M}.

By applying the HBM to the next family of potential systems
\begin{equation}\label{potgralfin}
\left\{\begin{array}{l}
\dot{x}=-y,\\
\dot{y}=x+\displaystyle\sum_{i=2}^M k_ix^i,
\end{array}\right.
\end{equation}
for $M=3,4,5,6,7,$ we obtain the corresponding $T_{1,M}(A)$ which
satisfy

$$
T_{1,M}(A)=2\,\pi + \left({k_2}^{2}-\frac{3}{4}\,{k_3} \right) \pi
{A}^{2}+O_M\left({A}^{3}\right).
$$
As can be seen, the quadratic terms do not depend on $M$. These
first terms only coincide with the corresponding ones of $T(A)$ when
$k_2=0$. Notice that this is the situation in Propositions
\ref{profinal} and \ref{proDuf}.

To get a more accurate approach of $T(A)$ we have applied the second
order HBM to \eqref{potgralfin} with $M=3$ obtaining
$$
T_2(A)=2\,\pi+\left(\frac{5}{6}\,{k_2}^{2}-\frac{3}{4}\,{k_3}
\right) \pi {A}^{2}+ O \left( {A}^{3} \right),
$$
result that coincides with the actual value of $T(A)$.

\section*{Conclusions}
Studying several examples  of potential systems we have seen that
the approximations $T_N(A)$ calculated using the $N$-th order HBM
keep some of the properties (analytic and qualitative) of the actual
period function $T(A)$. Moreover, this matching seems to improve
when $N$ increases.

We believe that obtaining general results to strengthen the above
relationship is a challenging question.

\subsection*{Acknowledgements} The two authors are  supported by the
MICIIN/FEDER  grant number MTM2008-03437 and the Generalitat de
Catalunya grant number 2009-SGR 410. The first author is also
supported by the grant AP2009-1189.

\end{document}